\documentclass{arxiv}

\usepackage{amssymb}
\usepackage{amsmath}
\usepackage{graphicx}  
\usepackage{epstopdf}
\usepackage{color}
\usepackage{bbold}
\usepackage{units}
\usepackage{algorithm}
\usepackage{enumitem}
\usepackage{pgfplots}
\usepackage{hyperref}

\newtheorem{theorem}{Theorem}

\newtheorem{lemma}{Lemma}
\newtheorem{remark}{Remark}

\newtheorem{proposition}{Proposition}
\newtheorem{definition}{Definition}
\newtheorem{example}{Example}

\newcommand{\tr}{\intercal}
\newcommand{\one}{\mathbf{1}}
\newcommand{\nf}{\mathrm{f}}

\newcommand{\nv}{\mathrm{v}}
\newcommand{\cvx}[1]{\mathrm{conv}\left(#1\right)}

\newcommand{\interior}[1]{\mathbf{int}\left(#1\right)}
\newcommand{\defeq}{:=}

\newcommand{\Poly}{\mathbb P}

\newcommand{\diam}{\mathrm{diam}}

\begin{document}

\begin{frontmatter}

\title{On Universal Minkowski Classes of Polytopes: \\ Theory and Applications}

\author[ST,IMT]{Boris Houska},
\author[LUH]{Matthias A.~M\"uller},
\author[IMT]{Mario Eduardo Villanueva}
\address[ST]{ShanghaiTech University, China}
\address[LUH]{Leibniz University Hannover, Germany}
\address[IMT]{IMT School for Advanced Studies Lucca, Italy}

\begin{keyword}
Linear Systems, Polytopic Control, Convex Optimization
\end{keyword}

\begin{abstract}
This paper establishes the existence, in every dimension, of universal Minkowski classes of configuration-constrained polytopes, with direct
applications to set-theoretic control. For any prescribed accuracy, such a class contains an approximation of every compact convex set to that accuracy and admits a jointly affine parameterization of facets and vertices. The construction is explicit and relies on a combinatorial averaging procedure that produces minimal configuration enclosures. In control, these polytopic templates enable convex reformulations of otherwise non-convex problems, including invariant-set computation and control Lyapunov function design. Tutorial examples illustrate both the geometric construction and its control applications.
\end{abstract}

\end{frontmatter}

\section{Introduction}

Set-theoretic methods are widely used in the analysis and control of constrained dynamical systems~\cite{Blanchini2008}. Many of these methods rely on polyhedral computing~\cite{Fukuda2020,Houska2025}, since polytopes admit finite-dimensional representations, interact naturally with linear and convex optimization, and can approximate arbitrary compact convex sets to any prescribed accuracy~\cite{Schneider1981}. In principle, modern optimization software makes it possible to manipulate polytopes with very large numbers of facets or vertices. Nevertheless, the optimization of a polytope itself remains difficult whenever both its facet and vertex representations enter the problem simultaneously.

This difficulty is fundamental in several set-theoretic control problems. Examples include the computation of robust control invariant sets, the construction of piecewise affine control Lyapunov functions, and the optimization of polytopic tubes in model predictive control~\cite{Gutman1986,Gutman1987,Villanueva2024,HouskaCLF2026,Houska2025}. In such problems, one would often like to parameterize a candidate polytope through its facet offsets, while dynamical conditions must be imposed at its vertices. However, the vertices generally depend nonlinearly and combinatorially on the facet offsets. Consequently, even when the underlying system is linear, the resulting optimization problem is typically non-convex.

\input{figures/fig1}

For example, consider the computation of a robust control invariant set for a discrete-time linear system subject to additive disturbances. A natural candidate set can be described by fixing its facet normals and optimizing the corresponding offsets. The invariance conditions, however, must be verified at the vertices of the candidate polytope, whose locations and incidence structure vary with the offsets. The feasible set of offsets is therefore generally non-convex. For uncontrolled systems, this difficulty can sometimes be circumvented through specialized constructions~\cite{Trodden2016}; in the controlled case, the simultaneous dependence on facets, vertices, and control actions remains a central obstacle.

Configuration-constrained polytopes provide a way to separate this geometric difficulty from the subsequent control optimization~\cite{Villanueva2024,HouskaCLF2026}. Their defining feature is that the relevant facet--vertex incidence structure is fixed in advance. Within a prescribed configuration, both the facet offsets and the vertex coordinates admit a joint affine parameterization. As a result, conditions involving both representations can be expressed through convex constraints. Once an appropriate configuration has been selected, a number of otherwise non-convex set-theoretic control problems can therefore be reformulated as convex optimization problems.

This observation raises a fundamental approximation question. Fixing the configuration improves tractability, but it also restricts the admissible polytope geometries. It is therefore not immediate whether configuration-constrained polytopes retain enough flexibility to approximate arbitrary compact convex sets. Without such an approximation property, the restriction to a prescribed configuration could introduce an irreducible geometric conservatism, independently of the efficiency with which the resulting optimization problem is solved.

The main purpose of this paper is to resolve this question. In particular, it resolves the geometric approximation issue left open by the configuration-constrained tube-MPC and control-Lyapunov-function constructions in~\cite{Villanueva2024,HouskaCLF2026}. We establish the existence, in every dimension, of \emph{universal Minkowski classes} of configuration-constrained polytopes. More precisely, for every prescribed accuracy, we show that there exists a Minkowski class whose members admit a joint affine facet--vertex parameterization and approximate every compact convex set of unit diameter to within that accuracy. Thus, the geometric flexibility required for universal approximation is compatible with the fixed combinatorial structure needed for convex optimization.

\subsection{Background and Related Work}

The mathematical foundations of polyhedral theory originate in the work of Minkowski~\cite{Minkowski1989}, Weyl~\cite{Weyl1934}, Carath\'eodory~\cite{Caratheodory1911}, and Farkas~\cite{Farkas1902}, among others. Subsequent developments revealed deep connections between convex polytopes, linear programming, combinatorics, and algebraic geometry~\cite{Gale1960,Dantzig1963,McMullen1989}. Standard references on the geometric and combinatorial theory of polytopes include the monographs by Gr\"unbaum~\cite{Gruenbaum1967} and Ziegler~\cite{Ziegler1995}.

From a computational perspective, the conversion between facet and vertex representations is a central operation in polyhedral computing~\cite{Graham1990,Fukuda2020}. Classical approaches include the double-description method~\cite{Motzkin1953}, beneath-and-beyond procedures~\cite{Seidel1981}, and later refinements for vertex and facet enumeration~\cite{Fukuda1996}. Reverse-search methods introduced by Avis and co-workers provide an alternative framework for enumerating polyhedral objects while controlling memory requirements~\cite{Avis1992,Avis1996}. These tools are highly developed, but their purpose is generally to enumerate or manipulate the geometry of a given polyhedron. They do not, by themselves, remove the non-convex dependence between the facets and vertices of a polytope whose geometry is being optimized.

Polyhedral methods have played a similarly important role in control. Gutman and Cwickel derived foundational vertex-based characterizations of control invariance for constrained linear systems~\cite{Gutman1986,Gutman1987}, while Bitsoris developed related algebraic conditions for invariant polytopes~\cite{Bitsoris1988}. These results form part of the broader theory of set invariance developed and surveyed by Blanchini~\cite{Blanchini1999b,Blanchini2008}. More recently, Trodden showed that minimal robust positively invariant polytopes can be computed through convex optimization for an important class of uncontrolled uncertain systems~\cite{Trodden2016}. Polyhedral techniques are also central to the analysis of parametric optimization problems arising in model-based control~\cite{Jones2005,Herceg2013}.

In model predictive control~\cite{Rawlings2015}, polytopes are routinely used to describe state and input constraints, terminal sets, reachable sets, and uncertainty tubes. Different restrictions on the geometry of the tube cross-sections lead to rigid, homothetic, elastic, and configuration-constrained tube MPC formulations~\cite{Langson2004,Rakovic2013,Rakovic2016,Villanueva2024}. A recent survey of polyhedral methods in control, including their use in tube MPC and set-theoretic design, is provided in~\cite{Houska2025}. The present paper complements this literature by addressing a prior geometric question: whether the tractable class of configuration-constrained polytopes can, in principle, retain the universal approximation capabilities of unrestricted polytopes.

\subsection{Contributions and Scope}

The primary contribution of this paper is a foundational result in convex and polyhedral geometry. We prove that, for every dimension $d\geq 1$ and every accuracy $\epsilon>0$, there exists a Minkowski class of configuration-constrained $d$-polytopes that is $\epsilon$-dense in the family of compact convex subsets of $\mathbb R^d$ with diameter~$1$. This result, stated in Theorem~\ref{thm::dense}, shows that fixing a suitable facet--vertex configuration does not preclude approximation to arbitrary prescribed accuracy.

The proof is constructive. In particular, we introduce a combinatorial averaging procedure that generates minimal configuration enclosures and thereby produces the universal Minkowski classes used in the approximation theorem. The construction is presented in Algorithm~\ref{alg:me} and analyzed in Theorem~\ref{thm::enclosure}. Its role in this paper is primarily theoretical: it provides a finite and systematic construction establishing the existence of the required classes in every dimension. The procedure is deliberately general and is not intended as an optimized implementation. Its computational cost can be substantial, particularly in dimensions $d\geq 3$, and the development of specialized algorithms and data structures is left for future work.

The relevance to set-theoretic control arises after this offline construction has been completed. The resulting configuration-constrained templates combine two properties that are usually in tension: they admit a joint affine parameterization of facets and vertices, and they can approximate arbitrary compact convex sets to any prescribed accuracy. The first property enables convex reformulations of problems involving simultaneous facet and vertex constraints, while the second ensures that the use of such templates need not impose an intrinsic approximation barrier. In particular, the templates can be used in convex formulations for robust control invariant-set computation, control Lyapunov function design, and related polytopic optimization problems.

The construction is performed once, offline, for a given dimension and approximation accuracy. The resulting templates can then be stored and reused across different systems and control-design problems of the same dimension. Accordingly, the main contribution is not a new online control algorithm, nor a claim that the general construction is already computationally optimal. Rather, it is the constructive existence of universal configuration-constrained Minkowski classes and, through this result, a theoretical justification for using jointly parameterized polytopic templates in convex set-theoretic control formulations. The main developments are illustrated by means of tutorial examples.

\subsection{Notation}

Throughout this article, the standard terminology for sets is used. In particular, we use the following symbols:
\begin{itemize}
\item $\mathbb R^d$ denotes the $d$-dimensional Euclidean space, and
\item $\mathfrak C^d$ denotes the set of nonempty compact convex sets in $\mathbb R^d$.
\end{itemize}
Additionally, the notion of partially ordered sets (posets) is used to describe the structure of set collections, such as the faces of a polyhedron. We use the standard notational conventions as in~\cite[Def.~2.5]{Ziegler1995}. The Hausdorff distance, diameter, and Minkowski sum of $\mathbb X,\mathbb Y\in\mathfrak C^d$ are defined, as usual, by
\begin{align*}
d_\mathrm{H}(\mathbb X, \mathbb Y) &\defeq \max \left\{ \max_{x \in \mathbb X} \min_{y \in \mathbb Y} \| x-y\|, \max_{y \in \mathbb Y} \min_{x \in \mathbb X} \| x-y\| \right\}, \\[0.16cm]
\diam(\mathbb X) &\defeq \max_{x,y \in \mathbb X} \| x-y \|, \quad \text{and} \\[0.16cm]
\mathbb X+ \mathbb Y &\defeq \{ \ x+y \ \mid \ x \in \mathbb X, \ y \in \mathbb Y \ \}.
\end{align*}
Moreover, we recall the following standard definitions of Minkowski classes, see also~\cite[Chap.~3.5]{Schneider2013}.
\begin{definition}
\label{def::MinkowskiClass}
A set of compact convex sets $\mathcal{X}\subseteq\mathfrak{C}^d$ is
called a Minkowski class if it is closed in the Hausdorff metric and
satisfies
\[
    X+Y\in\mathcal{X}
    \qquad\text{and}\qquad
    \alpha X\in\mathcal{X}
\]
for all $X,Y\in\mathcal{X}$ and all $\alpha\geq 0$.
\end{definition}
Examples of Minkowski classes include the set $\mathfrak C^d$ itself, the class of centrally symmetric compact convex sets, and the set of zonoids; that is, the set of sets that can be approximated by a sum of line segments~\cite{Schneider1983}.
\begin{definition}
\label{def::dense}
Let $\mathfrak X,\mathfrak Y\subseteq\mathfrak C^d$ be two given set collections with $\mathfrak X \subseteq \mathfrak Y$. We say that $\mathfrak X$ is $\epsilon$-dense in $\mathfrak Y$, for a given $\epsilon > 0$, if
\[
\forall \mathbb Y \in \mathfrak Y, \ \exists \mathbb X \in \mathfrak X: \quad d_\mathrm{H}( \mathbb X, \mathbb Y) \ \leq \ \epsilon \cdot \diam( \mathbb Y).
\]
\end{definition}
The latter definition is important for analyzing generic numerical set approximation methods. Namely, in the context of approximating a set $\mathbb Y \in \mathfrak Y$ with a set $\mathbb X \in \mathfrak X$, the constant $\epsilon$ in the above definition can be interpreted as a relative numerical approximation accuracy. Alternatively, if one scales the set $\mathbb Y$ such that $\mathrm{diam}(\mathbb Y) = 1$, the constant $\epsilon$ can be interpreted as the absolute approximation accuracy.

Finally, the symbol $\one$ denotes an all-ones vector whose dimension is clear from context. For vectors $a_1,a_2,\ldots,a_n \in \mathbb R^d$, we use the notation
\[
\cvx{a_1,a_2,\ldots,a_n} \ \defeq \ \left\{ \sum_{i=1}^n \theta_i a_i \middle|
\ \theta \geq 0, \ \mathbf{1}^\tr \theta = 1 \right\}
\]
to denote the convex hull. This notation is also used for matrices, where $\cvx{A}$ denotes the convex hull of the columns of $A \in \mathbb R^{d \times n}$.

\section{Polytopes}
The goal of this section is to introduce the basic definitions and properties of general polytopes and their face configuration.

\subsection{Faces of Polytopes}
\label{sec::polytopes}
Throughout this paper, we are interested in polytopes of dimension $d$ with at most $n$ facets. Let
\[
\mathbb A(n,d) \, \defeq \, \left\{ \, A \in \mathbb R^{n \times d} \, \middle| \,
\begin{array}{l}
\forall i \in \{ 1, \ldots, n \}, \ A_i \neq 0, \\
0 \in \interior{\cvx{A^\tr}}, \\
\forall i,j\in\{1,\ldots,n\},\,  i\neq j,\\
\| A_i \| A_j \neq \| A_j \| A_i
\end{array}
\,\right\}
\]
denote the corresponding set of admissible facet matrices. The conditions require every row of $A$ to be nonzero and exclude positively collinear distinct rows. These restrictions involve no loss of generality when facet directions are prescribed, because zero normals and repeated positive directions can be removed. In this context, a polytope with facet matrix \(A \in \mathbb A(n,d)\) is a
(potentially empty) set of the form
\[
\Poly(A,b) \ \defeq \ \{ \ x \in \mathbb R^d \ \mid \ A x \leq b \ \},
\]
for some vector \(b \in \mathbb R^{n}\). In this context, the condition $0 \in \interior{\cvx{A^\tr}}$ in the definition of $\mathbb A(n,d)$ is exactly the boundedness condition: it says that the facet normals $A_1,\dots,A_n$ positively span $\mathbb R^d$, or equivalently that the recession cone $\{x\in\mathbb R^d \mid Ax\le 0\}$ is \(\{0\}\). Thus, $\Poly(A,b)$ has no nonzero direction in which it can be unbounded; whenever
it is nonempty, it is a bounded polyhedron, hence a polytope. As usual, we associate with each index set $\mathbb I \subseteq \{ 1, \ldots, n \}$ a face of $\Poly(A,b)$, defined by
\[
\mathbb F( \mathbb I,A,b ) \ \defeq \ \left\{ \ x \in \Poly(A,b) \ \middle| \ \forall i \in \mathbb I, \ A_i x = b_i \ \right\}.
\]
Faces of dimensions $0$, $1$, $k$, and $d-1$ are called vertices, edges, $k$-faces, and facets. The number of $k$-faces of $\Poly(A,b)$ is denoted by $\nf_k(A,b)$. They are the components of the so-called $\nf$-vector. Moreover, we recall that the faces of a polytope are a bounded and graded poset with respect to inclusion. It is called the face lattice of $\Poly(A,b)$. In particular, we have
\[
\mathbb I \ \subseteq \ \mathbb J \qquad \Longrightarrow \qquad \mathbb F(\mathbb I, A,b) \ \supseteq \ \mathbb F( \mathbb J, A,b).
\]
Generally, two polytopes are called combinatorially equivalent if their face lattices are isomorphic, see also~\cite[Chap.~2.2]{Ziegler1995} for an in-depth discussion. Similarly, the following definitions are commonly used in almost all the literature on polytopes, see~\cite{Gruenbaum1967} and~\cite{Ziegler1995} for details.

\begin{definition}
\label{def::simple}
A polytope in $\mathbb R^d$ is called simple if all of its vertices are in exactly $d$ facets.
\end{definition}

\begin{definition}
\label{def::minimal}
The facet representation $\mathbb P = \Poly(A,b)$ with $A \in \mathbb A(n,d)$ is called minimal, if $\nf_{d-1}(A,b) = n$.
\end{definition}

We further recall that if $\Poly(A,b)$ is a minimal representation of a simple polytope, then the polytopes $\Poly(A,c)$ and $\Poly(A,b)$ are combinatorially equivalent for all vectors $c$ in a sufficiently small open neighborhood of $b$, see~\cite[Section~4.5]{Gruenbaum1967},~\cite[Section~2.5]{Ziegler1995}, and compare also~\cite[Cor.~2]{Villanueva2024}. In other words, simple polytopes are locally configuration stable.

\subsection{\texorpdfstring{Configuration Posets}{Configuration Posets}}
\label{sec::configurations}
Throughout, we assume that $A \in \mathbb A(n,d)$ is a given admissible facet matrix. We associate with each index set $\mathbb I \subseteq \{ 1, \ldots, n \}$ a nonempty convex cone,
\begin{align}
\mathbb C(\mathbb I, A) & \ \defeq \ \left\{
\ b \in \mathbb R^n \ \middle|
\  \begin{array}{l}
\exists x \in \mathbb R^d: \ \forall i \in \mathbb I, \\
A x \leq b, \ A_i x = b_i
\end{array} \
\right\}. \notag 
\end{align}
In words, $\mathbb C(\mathbb I, A)$ denotes the set of all parameters $b$ for which the face $\mathbb F(\mathbb I,A,b)$ is nonempty. Note that $\mathbb C(\mathbb I,A)$ is a closed polyhedral cone, whose lineality space is given by
\[
\operatorname{lin}(\mathbb C(\mathbb I,A))=\mathrm{Im}(A).
\]
To verify the inclusion
$\operatorname{lin}(\mathbb C(\mathbb I,A))\subseteq \mathrm{Im}(A)$, pick any $v\in \operatorname{lin}(\mathbb C(\mathbb I,A))$, such that $v\in \mathbb C(\mathbb I,A)$ and $-v\in \mathbb C(\mathbb I,A)$.
Then we can write
\[
v=A x+s \quad \text{and} \quad -v=A y+t
\]
for some $s,t\geq 0$. Adding these two identities gives
\[
0=A(x+y)+s+t.
\]
Choosing $w\in\mathbb R_{++}^n$ with $A^\tr w=0$, which exists because $0\in\interior{\cvx{A^\tr}}$, yields $w^\tr(s+t)=0$ and, hence, $s=t=0$. Therefore, $v\in\mathrm{Im}(A)$; see also~\cite{Houska2025}. Conversely, if $v\in\mathrm{Im}(A)$, say $v=A x$, then
$v\in\mathbb C(\mathbb I,A)$ via $x$ and
$-v\in\mathbb C(\mathbb I,A)$ via $-x$. Hence
$v\in\operatorname{lin}(\mathbb C(\mathbb I,A))$. Therefore,
\[
\operatorname{lin}(\mathbb C(\mathbb I,A))=\mathrm{Im}(A).
\]
Moreover, the set $\mathbb C(\varnothing, A)$ corresponds to the set of parameters $b$ for which $\Poly(A,b) \neq \varnothing$. Additionally, we introduce the shorthand
\[
\overline{\mathbb C}(A) \ \defeq \ \bigcap_{i \in \{ 1, \ldots, n \}} \ \mathbb C( \{ i\}, A ).
\]
The set $\overline{\mathbb C}(A)$ is a convex cone and is the natural parameter domain for nonempty polytopes. Indeed, if $\Poly(A,b)\neq\varnothing$, the support vector $b'\in\mathbb R^n$ defined by
\[
b'_i \ \defeq \ \max_{x\in\Poly(A,b)}A_i x,\qquad i\in\{1,\ldots,n\},
\]
belongs to $\overline{\mathbb C}(A)$ and satisfies $\Poly(A,b)=\Poly(A,b')$.
These cones can be used to analyze polytope configurations through the index-set collection
\[
\mathfrak I(A,b) \ \defeq \ \{ \ \mathbb I \subseteq \{ 1, \ldots, n \} \ \mid \ \mathbb F( \mathbb I, A,b) \neq \varnothing \ \}.
\]
In words, $\mathfrak I(A,b)$ denotes the collection of all index sets that correspond to the nonempty faces of a given polytope $\Poly(A,b)$.
\begin{definition}
\label{def::cccone}
Let $A \in \mathbb A(n,d)$ and $b \in \overline{\mathbb C}(A)$ be given. The configuration cone of a polytope $\Poly(A,b)$ is defined as
\[
\mathbb K(A,b) \ \defeq \ \bigcap_{\mathbb I \in \mathfrak I(A,b)} \mathbb C(\mathbb I, A).
\]
Additionally, two polytopes $\Poly(A,b)$ and $\Poly(A,c)$ with parameters $b,c \in \overline{\mathbb C}(A)$ are called configuration-equivalent for the given matrix $A$, if $\mathbb K(A,b) = \mathbb K(A,c)$.
\end{definition}

\begin{remark}
\label{rem::normal-fans}
Let $\mathbb P\subseteq\mathbb R^d$ be a polytope and let $\mathbb F$ be a nonempty face of $\mathbb P$. Its normal cone is
\[
\left\{p\in\mathbb R^d\ \middle| \forall x\in\mathbb P, \ \forall y \in \mathbb F, \
p^\tr(x-y)\leq0 \ 
\right\}.
\]
The collection of the normal cones of all faces is the normal fan. A fan $\mathcal N_1$ refines a fan $\mathcal N_2$ if every cone of $\mathcal N_1$ is contained in a cone of $\mathcal N_2$; equivalently, $\mathcal N_2$ is a coarsening of $\mathcal N_1$. Thus, refinement subdivides normal cones, whereas coarsening merges adjacent normal cones. For fixed facet-normal labels, parameters in the relative interior of one configuration cone have the same labeled normal fan. Moving to its boundary can merge normal cones, create additional facet incidences, make inequalities redundant, or produce a lower-dimensional degeneration. These notions are standard in the deformation- and type-cone literature; see~\cite{Ziegler1995,Schneider2013},~\cite[Sec.~1.1 and Cor.~1.42]{FillastreIzmestiev2017}, and~\cite{Houska2025}.
\end{remark}

In the terminology of Remark~\ref{rem::normal-fans}, if $\Poly(A,b)$ is minimally represented, then $c\in\mathbb K(A,b)$ precisely when the labeled normal fan of $\Poly(A,b)$ refines that of $\Poly(A,c)$. The relative interior of $\mathbb K(A,b)$ therefore consists of parameters with the same labeled normal fan, whereas its boundary contains coarsenings, including cases with redundant inequalities or lower-dimensional degenerations. This is the configuration-cone form of the classical deformation/type-cone description; see~\cite[Thms.~8 and~10]{McMullen1973},~\cite[Cor.~1.42, Lem.~1.43, and Rem.~1.45]{FillastreIzmestiev2017}, and~\cite{Houska2025}.
Configuration-equivalent polytopes are necessarily combinatorially equivalent, but the converse need not hold; see Example~\ref{sec::examples1}. Moreover, configuration-equivalence is label-dependent. For example, if $\tilde A$ is obtained by permuting the rows of $A$, then $\Poly(\tilde A,\one)=\Poly(A,\one)$, although the indexed facet representation has changed. This dependence is natural in polyhedral computation, where the rows of the facet matrix and the corresponding face-index lists are stored explicitly.

The family of configuration cones is closed under finite intersections. In the type-cone description, the intersection of two chamber closures is either empty or a common face; whenever it is nonempty, that face is itself represented by the incidences active at a parameter in its relative interior. Thus,
\[
\mathbb K(A,b) \cap \mathbb K(A,c)
\]
is again a configuration cone. This construction and the definition below follow the type-cone-complex viewpoint of~\cite{McMullen1973}.
\begin{definition}
\label{def::cclattice}
The poset of configuration cones associated with a matrix $A \in \mathbb A(n,d)$,
\[
\mathfrak K(A) \ \defeq \ \left\{ \ \mathbb K(A,b) \ \middle| \ b \in \overline{\mathbb C}(A) \ \right\},
\]
partially ordered with respect to inclusion, is called the configuration poset of $A$.
\end{definition}
The poset $\mathfrak K(A)$ is finite, because only finitely many collections $\mathfrak I(A,b)$ can occur. Every configuration cone is a closed polyhedral cone with lineality space $\mathrm{Im}(A)$. Moreover, $\mathrm{Im}(A)$ is the unique minimal element: every configuration cone contains it, while $\mathbb K(A,Ay)=\mathrm{Im}(A)$ for every $y\in\mathbb R^d$. Here, $\dim(\mathrm{Im}(A))=d$ because $0\in\interior{\cvx{A^\tr}}$ implies $\operatorname{rank}(A)=d$. In contrast to a lattice, $\mathfrak K(A)$ may have several maximal elements. We denote their number by $\sigma(A)\in\mathbb N$.

Every maximal element $\mathbb M_i$ has nonempty interior and can be written as
\[
\mathbb M_i=\mathbb K(A,b_i),\qquad b_i\in\interior{\mathbb M_i},\qquad i\in\{1,\ldots,\sigma(A)\}.
\]
The representative $b_i$ can be chosen so that $\Poly(A,b_i)$ is a simple polytope with a minimal facet representation. These standard deformation/type-cone properties follow from~\cite{McMullen1973,FillastreIzmestiev2017,Villanueva2024,Houska2025}.

\subsection{Configuration-Constrained Polytopes}
Configuration-constrained polytopes were originally introduced in~\cite{Villanueva2024}. Using the notation of the current paper, they are defined as follows.

\begin{definition}
Let $A \in \mathbb A(n,d)$ be a given facet matrix. For any $b \in \overline{\mathbb C}(A)$, the associated class of configuration-constrained polytopes is
\[
\mathfrak P(A,b) \ \defeq \ \left\{ \ \Poly(A,c) \ \mid \ c \in \mathbb K(A,b) \ \right\}.
\]
\end{definition}
The class $\mathfrak P(A,b)$ is nonempty because $b\in\mathbb K(A,b)$. It may, however, contain no simple polytopes unless $b$ represents a maximal element of $\mathfrak K(A)$. Regardless of simplicity, $\mathfrak P(A,b)$ admits a double description: there exist vertex matrices $V_1,V_2,\ldots,V_{\nv}\in\mathbb R^{d\times n}$ such that
\begin{equation}\label{eq::vertices}
\mathfrak P(A,b) \ = \ \left\{ \ \cvx{V_1 c, \ldots,V_{\nv} c} \ \mid \ c \in \mathbb K(A,b) \ \right\}.
\end{equation}
This was proved for maximal configuration cones in~\cite[Thm.~2]{Villanueva2024} and extended to arbitrary configuration cones in~\cite[Prop.~1]{Houska2024}.

The following technical lemma summarizes a simple but relevant consequence of these results.

\begin{lemma}
\label{lem::MinkowskiClass}
Let $A \in \mathbb A(n,d)$ and $b \in \overline{\mathbb C}(A)$ be given. Then $\mathfrak P(A,b)$ is a Minkowski class (see Definition~\ref{def::MinkowskiClass}).
\end{lemma}

\textit{Proof.} We first verify closedness. Let $\Poly(A,c^\ell)\in\mathfrak P(A,b)$ converge in the Hausdorff metric to $\mathbb P\in\mathfrak C^d$. Since $c^\ell\in\mathbb K(A,b)\subseteq\overline{\mathbb C}(A)$, every row of $A$ is supporting and
\[
c_i^\ell=\max_{x\in\Poly(A,c^\ell)}A_i x,\qquad i\in\{1,\ldots,n\}.
\]
Continuity of support functions under Hausdorff convergence implies $c^\ell\to c$, where $c_i=\max_{x\in\mathbb P}A_i x$. The cone $\mathbb K(A,b)$ is closed, so $c\in\mathbb K(A,b)$. Passing to the limit in the common vertex representation~\eqref{eq::vertices} yields $\mathbb P=\Poly(A,c)\in\mathfrak P(A,b)$.

It remains to prove closure under Minkowski addition and nonnegative
scaling. Let $c_1,c_2 \in \mathbb K(A,b)$. Since $\mathbb K(A,b)$ is a convex cone, $c_1+c_2\in\mathbb K(A,b)$. For every $i\in\{1,\ldots,n\}$,
\begin{eqnarray*}
\max_{x \in \Poly(A,c_1) + \Poly(A,c_2)} A_i x &\leq& \max_{x_1 \in \Poly(A,c_1)} A_i x_1 + \max_{x_2 \in \Poly(A,c_2)} A_i x_2 \\
&=& (c_1+c_2)_i,
\end{eqnarray*}
which implies
\[
\Poly(A,c_1) + \Poly(A,c_2) \ \subseteq \ \Poly(A,c_1+c_2).
\]
Conversely, the common vertex matrices in~\eqref{eq::vertices} imply that every $x\in\Poly(A,c_1+c_2)$ can be written as
\begin{eqnarray*}
x &=& \sum_{i=1}^{\nv}\lambda_iV_i(c_1+c_2) \\
&=& \sum_{i=1}^{\nv}\lambda_iV_ic_1+\sum_{i=1}^{\nv}\lambda_iV_ic_2
\ \in\ \Poly(A,c_1)+\Poly(A,c_2),
\end{eqnarray*}
for suitable $\lambda_i\geq0$ with $\sum_i\lambda_i=1$. Hence,
\[
\Poly(A,c_1)+\Poly(A,c_2)=\Poly(A,c_1+c_2).
\]
Finally, since $\mathbb K(A,b)$ is a cone, $\alpha c\in \mathbb K(A,b)$ for every
$c\in \mathbb K(A,b)$ and every $\alpha\geq 0$. The same representation
therefore gives
\[
    \alpha \Poly(A,c)=\Poly(A,\alpha c)\in\mathfrak{P}(A,b),
\]
which proves closure under nonnegative scaling.
\qed

\section{Configuration Enclosures}
\label{sec::monadic}
While configuration-constrained polytopes provide a practical joint facet--vertex representation, it remains to determine whether such classes are sufficiently flexible to approximate arbitrary compact convex sets. The primary objective of this paper is to show that suitable configuration-constrained classes do possess this approximation property. Our main result can be formulated as follows.

\begin{theorem}
\label{thm::dense}
There exists for every $\epsilon > 0$ and every dimension $d \geq 1$ an integer $n \geq 1$ and a matrix $A \in \mathbb A(n,d)$ for which the Minkowski class $\mathfrak P(A,\one)$ is $\epsilon$-dense in $\mathfrak C^d$.
\end{theorem}

The proof of Theorem~\ref{thm::dense} is given in Section~\ref{sec::proof1}. We first introduce the concept of configuration enclosures.

\subsection{Outline of the Main Technical Idea}
The main technical idea for proving Theorem~\ref{thm::dense} is to start with an arbitrary facet matrix $A \in \mathbb A(n,d)$ and replace it with a configuration enclosure, defined as follows.
\begin{definition}
A facet matrix $B \in \mathbb A(m,d)$ with $\one\in\overline{\mathbb C}(B)$ is called a configuration enclosure of $A \in \mathbb A(n,d)$ if
\[
\forall b\in\mathbb{C}(\varnothing,A),\quad
\exists c\in\mathbb K(B,\one):\quad \Poly(A,b)=\Poly(B,c).
\]
\end{definition}
A given matrix $A$ may have many configuration enclosures; the enclosure matrix $B$ has $m$ facet normals. As follows from the normal-fan argument below, an enclosure must contain the original normal directions up to positive scaling, and hence $m\geq n$. In the context of practical applications, however, we are interested in keeping the number of new facet normals, $m-n$, as small as possible. Therefore, the following definition of minimal configuration enclosures is useful.
\begin{definition}
We call $B \in \mathbb A(m,d)$ a minimal configuration enclosure of $A \in \mathbb A(n,d)$ if there is no other configuration enclosure of $A$ that has fewer rows than $B$.
\end{definition}
Given these definitions, our strategy for proving Theorem~\ref{thm::dense} is to (i) recall the well-known fact that general polytopes can approximate arbitrary compact convex sets arbitrarily well~\cite{Schneider1981} and (ii) replace such arbitrary polytopes by their configuration enclosures.

\subsection{Monotypic Polytopes}
It is useful to distinguish facet-normal configurations for which one maximal configuration class already contains every polytope representable with those normal directions. We call such configurations monotypic.
\begin{definition}\label{def::monotypeA}
We call a matrix $A \in \mathbb A(n,d)$ monotypic if it satisfies $\sigma(A)=1$.
\end{definition}
In words, monotypic polytopes are precisely those whose configuration poset has a unique maximal element. A polytope $\Poly(A,b)$ is called monotypic if $b$ represents this unique maximal element $\mathbb M$, that is, $\mathbb M=\mathbb K(A,b)$. Notice that Definition~\ref{def::monotypeA} is equivalent to the notion of monotype introduced in~\cite{McMullen1974} relative to the set of normals of induced by the rows of $A$. 

\begin{remark}
Monotypy should be distinguished from strong monotypy, which requires the facet normals to determine the combinatorial type of the complete affine arrangement of the facet hyperplanes and therefore implies monotypy.
Strong monotypy is equivalent to the generating property: every nonempty
intersection of translates of the polytope is a Minkowski summand of the
polytope~\cite{Bui2025,Mcmullen2026}.
\end{remark}

Some important examples of monotypic polytopes are collected in the proposition below.
\begin{proposition}
\label{prop::monadic}
Every matrix $A\in\mathbb A(n,d)$ is monotypic in each of the following cases:
\begin{itemize}
\item $d = 1$ and $n = 2$, or
\item $d = 2$ and $n \geq 3$, or
\item $d \geq 3$ and $n = d+1$, or
\item $d \geq 3$, $n=2d$, and $A = [I_d\ -I_d]^\tr$.
\end{itemize}
The corresponding monotypic polytopes are, respectively, all intervals with nonempty interior in $\mathbb R^1$, all simple polytopes in $\mathbb R^2$, all simplices with nonempty interior in any dimension, and all axis-aligned rectangular boxes with nonempty interior in any dimension.
\end{proposition}

\textit{Proof}
The case $d=1$ and $n=2$ is trivial, as all intervals in $\mathbb R^1$ with nonempty interior are configuration equivalent. Moreover, the fact that all simple polytopes in $\mathbb R^2$ are monotypic follows from~\cite[Rem.~3]{Villanueva2024}. Polytopes with $n=d+1$ facets and nonempty interior in $\mathbb R^d$ are necessarily simplices, which can only have one possible configuration: any $d$ facets have a joint vertex. Finally, axis-aligned boxes are monotypic because every such $d$-box is a direct product of $d$ one-dimensional intervals, reducing this case to the one-dimensional setting.
\qed

Unfortunately, as much as it is remarkable that all simple polytopes in $\mathbb R^2$ are monotypic, it turns out that the situation in dimension $d \geq 3$ is very different. Namely, in dimensions $d\geq 3$, however, monotypy imposes nontrivial
restrictions on the facet-normal configuration, and arbitrary
normal configurations need not be monotypic. The following minimal
example illustrates this obstruction.

\begin{example}[Minimal non-monotypic polytope]
\label{sec::examples1}
Given Proposition~\ref{prop::monadic}, the simplest possible non-monotypic facet matrix must have $d \geq 3$ columns and $n \geq d+2 \geq 5$ rows, because simplices are monotypic. One such matrix with the minimum possible numbers of rows and columns is
\[
A \ = \ \left(
\begin{array}{rrrrr}
0 & \ -1 & \ \phantom{-}1 &  0 & 0 \\
0 &  0 & 0 & \ -1 & \ \phantom{-}1 \\
-1 & 1 & 1 &  1 & 1
\end{array}
\right)^\tr.
\]
Namely, for this matrix, the sets $\Poly(A,b_1)$ and $\Poly(A,b_2)$ with
\begin{align*}
b_1 & \ = \ \left( \ 0, \ 3, \ 3, \ 1, \ 1 \ \right)^\tr \\
\text{and} \qquad b_2 & \ = \ \left( \ 0, \ 1, \ 1, \ 3, \ 3 \ \right)^\tr
\end{align*}
are both simple polytopes, but they are not configuration-equivalent~\footnote{As already mentioned earlier, one should distinguish carefully between the notions of combinatorial equivalence and configuration equivalence. In this example, the polytope $\Poly(A,b_1)$ can be obtained by rotating the polytope $\Poly(A,b_2)$ by $90^\circ$ around the $z$-axis. Thus, these polytopes are combinatorially equivalent. They are not configuration-equivalent because the row labels of $A$ are fixed.}, as visualized below:\\
\makebox[\linewidth]{
\begin{tabular}{cc}

\begin{minipage}{0.2\textwidth}
\begin{center}
\[
\Poly(A,b_1)
\]
\begin{tikzpicture}%
	[x={(-0.546345cm, -0.225279cm)},
	y={(0.837560cm, -0.146890cm)},
	z={(-0.000053cm, 0.963157cm)},
	scale=0.57000,
	back/.style={dotted, thin},
	edge/.style={color=black, thick},
	facet/.style={fill=black,fill opacity=0.050000},
	vertex/.style={inner sep=0.1pt,circle,draw=black!25!black,fill=black!75!black,thick}]
\coordinate (-1.00000, -3.00000, 0.00000) at (-1.00000, -3.00000, 0.00000);
\coordinate (-1.00000, 3.00000, 0.00000) at (-1.00000, 3.00000, 0.00000);
\coordinate (0.00000, -2.00000, 1.00000) at (0.00000, -2.00000, 1.00000);
\coordinate (0.00000, 2.00000, 1.00000) at (0.00000, 2.00000, 1.00000);
\coordinate (1.00000, -3.00000, 0.00000) at (1.00000, -3.00000, 0.00000);
\coordinate (1.00000, 3.00000, 0.00000) at (1.00000, 3.00000, 0.00000);
\draw[edge,back] (-1.00000, -3.00000, 0.00000) -- (-1.00000, 3.00000, 0.00000);
\draw[edge,back] (-1.00000, -3.00000, 0.00000) -- (0.00000, -2.00000, 1.00000);
\draw[edge,back] (-1.00000, -3.00000, 0.00000) -- (1.00000, -3.00000, 0.00000);
\node[vertex] at (-1.00000, -3.00000, 0.00000)     {};
\fill[facet] (1.00000, 3.00000, 0.00000) -- (0.00000, 2.00000, 1.00000) -- (0.00000, -2.00000, 1.00000) -- (1.00000, -3.00000, 0.00000) -- cycle {};
\fill[facet] (1.00000, 3.00000, 0.00000) -- (-1.00000, 3.00000, 0.00000) -- (0.00000, 2.00000, 1.00000) -- cycle {};
\draw[edge] (-1.00000, 3.00000, 0.00000) -- (0.00000, 2.00000, 1.00000);
\draw[edge] (-1.00000, 3.00000, 0.00000) -- (1.00000, 3.00000, 0.00000);
\draw[edge] (0.00000, -2.00000, 1.00000) -- (0.00000, 2.00000, 1.00000);
\draw[edge] (0.00000, -2.00000, 1.00000) -- (1.00000, -3.00000, 0.00000);
\draw[edge] (0.00000, 2.00000, 1.00000) -- (1.00000, 3.00000, 0.00000);
\draw[edge] (1.00000, -3.00000, 0.00000) -- (1.00000, 3.00000, 0.00000);
\node[vertex] at (-1.00000, 3.00000, 0.00000)     {};
\node[vertex] at (0.00000, -2.00000, 1.00000)     {};
\node[vertex] at (0.00000, 2.00000, 1.00000)     {};
\node[vertex] at (1.00000, -3.00000, 0.00000)     {};
\node[vertex] at (1.00000, 3.00000, 0.00000)     {};
\end{tikzpicture}
\end{center}
\end{minipage}

&

\begin{minipage}{0.2\textwidth}
    
\[
\Poly(A,b_2)
\]

\begin{center}

\begin{tikzpicture}%
	[x={(-0.546345cm, -0.225279cm)},
	y={(0.837560cm, -0.146890cm)},
	z={(-0.000053cm, 0.963157cm)},
	scale=0.570000,
	back/.style={dotted, thin},
	edge/.style={color=black, thick},
	facet/.style={fill=black,fill opacity=0.050000},
	vertex/.style={inner sep=0.1pt,circle,draw=black!25!black,fill=black!75!black,thick}]
\coordinate (-3.00000, -1.00000, 0.00000) at (-3.00000, -1.00000, 0.00000);
\coordinate (-3.00000, 1.00000, 0.00000) at (-3.00000, 1.00000, 0.00000);
\coordinate (-2.00000, 0.00000, 1.00000) at (-2.00000, 0.00000, 1.00000);
\coordinate (2.00000, 0.00000, 1.00000) at (2.00000, 0.00000, 1.00000);
\coordinate (3.00000, -1.00000, 0.00000) at (3.00000, -1.00000, 0.00000);
\coordinate (3.00000, 1.00000, 0.00000) at (3.00000, 1.00000, 0.00000);
\draw[edge,back] (-3.00000, -1.00000, 0.00000) -- (-3.00000, 1.00000, 0.00000);
\draw[edge,back] (-3.00000, -1.00000, 0.00000) -- (-2.00000, 0.00000, 1.00000);
\draw[edge,back] (-3.00000, -1.00000, 0.00000) -- (3.00000, -1.00000, 0.00000);
\node[vertex] at (-3.00000, -1.00000, 0.00000)     {};
\fill[facet] (3.00000, 1.00000, 0.00000) -- (-3.00000, 1.00000, 0.00000) -- (-2.00000, 0.00000, 1.00000) -- (2.00000, 0.00000, 1.00000) -- cycle {};
\fill[facet] (3.00000, 1.00000, 0.00000) -- (2.00000, 0.00000, 1.00000) -- (3.00000, -1.00000, 0.00000) -- cycle {};
\draw[edge] (-3.00000, 1.00000, 0.00000) -- (-2.00000, 0.00000, 1.00000);
\draw[edge] (-3.00000, 1.00000, 0.00000) -- (3.00000, 1.00000, 0.00000);
\draw[edge] (-2.00000, 0.00000, 1.00000) -- (2.00000, 0.00000, 1.00000);
\draw[edge] (2.00000, 0.00000, 1.00000) -- (3.00000, -1.00000, 0.00000);
\draw[edge] (2.00000, 0.00000, 1.00000) -- (3.00000, 1.00000, 0.00000);
\draw[edge] (3.00000, -1.00000, 0.00000) -- (3.00000, 1.00000, 0.00000);
\node[vertex] at (-3.00000, 1.00000, 0.00000)     {};
\node[vertex] at (-2.00000, 0.00000, 1.00000)     {};
\node[vertex] at (2.00000, 0.00000, 1.00000)     {};
\node[vertex] at (3.00000, -1.00000, 0.00000)     {};
\node[vertex] at (3.00000, 1.00000, 0.00000)     {};
\end{tikzpicture} .
\end{center}
\end{minipage}

\end{tabular}}
\medskip
The second and third facets of the left polytope have no joint vertex, whereas the second and third facets of the right polytope have two joint vertices. Of course, for such simple examples, we can still use our ``geometric intuition'' to construct a configuration enclosure. For instance, we can add one row to $A$ to construct the new matrix
\[
B \ = \ \left(
\begin{array}{rrrrrr}
0 & \ -1 & \ \phantom{-}1 &  0 & 0 & 0 \\
0 &  0 & 0 & \ -1 & \ \phantom{-}1 & 0 \\
-1 & 1 & 1 &  1 & 1 & \phantom{-}1
\end{array}
\right)^\tr,
\]
which turns out to be a minimal configuration enclosure of $A$. In fact, $B$ is even monotypic. For example, an associated monotypic polytope $\Poly(B,c)$ is obtained for $c \ = \ \left( 0, 2, 2, 2, 2, 1 \right)^\tr$,
\begin{center}
\begin{tabular}{cl}
\begin{minipage}{0.15\textwidth}
\vspace{-1.4cm}
$\Poly(B,c) \ = \ $
\end{minipage}
&
\hspace{-1cm}
\begin{tikzpicture}%
	[x={(-0.546345cm, -0.225279cm)},
	y={(0.837560cm, -0.146890cm)},
	z={(-0.000053cm, 0.963157cm)},
	scale=0.600000,
	back/.style={dotted, thin},
	edge/.style={color=black, thick},
	facet/.style={fill=black,fill opacity=0.050000},
	vertex/.style={inner sep=0.1pt,circle,draw=black!15!black,fill=black!75!black,thin}]
\coordinate (-2.00000, -2.00000, 0.00000) at (-2.00000, -2.00000, 0.00000);
\coordinate (-2.00000, 2.00000, 0.00000) at (-2.00000, 2.00000, 0.00000);
\coordinate (-1.00000, -1.00000, 1.00000) at (-1.00000, -1.00000, 1.00000);
\coordinate (-1.00000, 1.00000, 1.00000) at (-1.00000, 1.00000, 1.00000);
\coordinate (1.00000, -1.00000, 1.00000) at (1.00000, -1.00000, 1.00000);
\coordinate (1.00000, 1.00000, 1.00000) at (1.00000, 1.00000, 1.00000);
\coordinate (2.00000, -2.00000, 0.00000) at (2.00000, -2.00000, 0.00000);
\coordinate (2.00000, 2.00000, 0.00000) at (2.00000, 2.00000, 0.00000);
\draw[edge,back] (-2.00000, -2.00000, 0.00000) -- (-2.00000, 2.00000, 0.00000);
\draw[edge,back] (-2.00000, -2.00000, 0.00000) -- (-1.00000, -1.00000, 1.00000);
\draw[edge,back] (-2.00000, -2.00000, 0.00000) -- (2.00000, -2.00000, 0.00000);
\node[vertex] at (-2.00000, -2.00000, 0.00000)     {};
\fill[facet] (2.00000, 2.00000, 0.00000) -- (-2.00000, 2.00000, 0.00000) -- (-1.00000, 1.00000, 1.00000) -- (1.00000, 1.00000, 1.00000) -- cycle {};
\fill[facet] (2.00000, 2.00000, 0.00000) -- (1.00000, 1.00000, 1.00000) -- (1.00000, -1.00000, 1.00000) -- (2.00000, -2.00000, 0.00000) -- cycle {};
\fill[facet] (1.00000, 1.00000, 1.00000) -- (-1.00000, 1.00000, 1.00000) -- (-1.00000, -1.00000, 1.00000) -- (1.00000, -1.00000, 1.00000) -- cycle {};
\draw[edge] (-2.00000, 2.00000, 0.00000) -- (-1.00000, 1.00000, 1.00000);
\draw[edge] (-2.00000, 2.00000, 0.00000) -- (2.00000, 2.00000, 0.00000);
\draw[edge] (-1.00000, -1.00000, 1.00000) -- (-1.00000, 1.00000, 1.00000);
\draw[edge] (-1.00000, -1.00000, 1.00000) -- (1.00000, -1.00000, 1.00000);
\draw[edge] (-1.00000, 1.00000, 1.00000) -- (1.00000, 1.00000, 1.00000);
\draw[edge] (1.00000, -1.00000, 1.00000) -- (1.00000, 1.00000, 1.00000);
\draw[edge] (1.00000, -1.00000, 1.00000) -- (2.00000, -2.00000, 0.00000);
\draw[edge] (1.00000, 1.00000, 1.00000) -- (2.00000, 2.00000, 0.00000);
\draw[edge] (2.00000, -2.00000, 0.00000) -- (2.00000, 2.00000, 0.00000);
\node[vertex] at (-2.00000, 2.00000, 0.00000)     {};
\node[vertex] at (-1.00000, -1.00000, 1.00000)     {};
\node[vertex] at (-1.00000, 1.00000, 1.00000)     {};
\node[vertex] at (1.00000, -1.00000, 1.00000)     {};
\node[vertex] at (1.00000, 1.00000, 1.00000)     {};
\node[vertex] at (2.00000, -2.00000, 0.00000)     {};
\node[vertex] at (2.00000, 2.00000, 0.00000)     {};
\end{tikzpicture}.
\end{tabular}
\end{center}
In addition to being monotypic, this polytope represents the Minkowski average
\[
\Poly \left( B, c \right) \ = \ \frac{\Poly( A, b_1 ) + \Poly( A, b_2 )}{2}.
\]
This example anticipates the Minkowski-averaging construction of minimal configuration enclosures developed next.
\end{example}

\section{Combinatorial Averaging}
This section develops a general method for constructing configuration enclosures and thereby provides a constructive proof of Theorem~\ref{thm::dense}.

\subsection{Minkowski Sums}
\label{sec::MinkowskiSums}
Let $\mathbb P_1,\ldots,\mathbb P_s\subseteq\mathbb R^d$ be full-dimensional polytopes that admit a (potentially nonminimal) facet representation $\mathbb P_i=\Poly(A,b_i)$, $i\in\{1,\ldots,s\}$. For $\lambda\in\mathbb R_{++}^s$, consider the weighted Minkowski average
\[
\overline{\mathbb P}_\lambda \ \defeq \ \frac{1}{s} \sum_{i=1}^s \lambda_i \mathbb P_i.
\]
The normal fan of a Minkowski sum is the common refinement of the normal fans of its summands. Consequently, the normal fan of $\overline{\mathbb P}_\lambda$ is independent of the positive weights $\lambda$; see~\cite[Sec.~6]{McMullen1973} and~\cite[Sec.~3.1]{Schneider2013}. Choose $z\in\interior{\overline{\mathbb P}_\one}$. After translating by $-z$ and positively scaling the inequalities in a minimal facet representation, we may write
\begin{align}
\label{eq::P1}
\overline{\mathbb P}_\one-z \ = \ \Poly\left(B,\one\right).
\end{align}
This normalization changes neither the normal fan nor the number of facets. The common-refinement property and the normal-fan characterization of configuration cones then imply
\begin{align}
\label{eq::aux1}
\forall \lambda\in\mathbb R_{++}^s,\ \exists c\in\mathbb K(B,\one):\quad
\overline{\mathbb P}_\lambda=\Poly(B,c).
\end{align}
The following proposition records the corresponding boundary cases.

\begin{proposition}
\label{prop::average}
Let $\mathbb P_i = \Poly(A,b_i)$, $i \in \{ 1, \ldots, s \}$, be full-dimensional polytopes with $A \in \mathbb A(n,d)$, and let $B \in \mathbb A(m,d)$ be a normalized minimal facet matrix satisfying~\eqref{eq::P1}. Then there exist vectors $c_i \in \mathbb K(B,\one)$ such that
\[
\mathbb P_i=\Poly(A,b_i)=\Poly(B,c_i),
\qquad i\in\{1,\ldots,s\}.
\]
\end{proposition}
\textit{Proof.} The normal fan of $\Poly(B,\one)$ is the common refinement of the normal fans of the summands and therefore refines the normal fan of each $\mathbb P_i$. Define the $B$-support vector of $\mathbb P_i$ by
\[
(c_i)_j\defeq\max_{x\in\mathbb P_i}B_jx,\qquad j\in\{1,\ldots,m\}.
\]
Every facet-normal ray of $\mathbb P_i$ is a ray of the common refinement, so $\Poly(B,c_i)=\mathbb P_i$. The normal-fan characterization of configuration cones further gives $c_i\in\mathbb K(B,\one)$; see~\cite[Thm.~8]{McMullen1973} and~\cite{Houska2025}.
\qed

\subsection{Combinatorial Averaging}
The core strategy for constructing configuration enclosures of an arbitrary matrix $A \in \mathbb A(n,d)$ is to construct the Minkowski average of its configurations. The corresponding method is summarized in Algorithm~\ref{alg:me}.
\begin{algorithm}[ht!]
\caption{Combinatorial Averaging}
\label{alg:me}
\smallskip
\begin{enumerate}[label=\arabic*)]
\addtolength{\itemsep}{2pt}
\item Start with any matrix $A \in \mathbb A(n,d)$.
\item \label{step::2} Pick for each maximal element $\mathbb M_i \in \mathfrak K(A)$ a representative $b_i\in\interior{\mathbb M_i}$ such that $\mathbb M_i = \mathbb K(A,b_i)$.
\item Form the Minkowski average
\[
\overline{\mathbb P}\defeq\frac{1}{\sigma(A)}\sum_{i=1}^{\sigma(A)}\Poly(A,b_i),
\]
choose $z\in\interior{\overline{\mathbb P}}$, and compute a normalized minimal facet representation
\[
\Poly(B,\one)=\overline{\mathbb P}-z.
\]
\item Return the matrix $B$.
\end{enumerate}
\smallskip
\end{algorithm}
The following theorem establishes the fact that this procedure returns the minimal configuration enclosure of any given facet matrix.

\begin{theorem}
\label{thm::enclosure}
Algorithm~\ref{alg:me} returns a minimal configuration enclosure of any given input matrix $A \in \mathbb A(n,d)$.
\end{theorem}
\textit{Proof.} Let $\mathbb P_i\defeq\Poly(A,b_i)$, $i\in\{1,\ldots,\sigma(A)\}$, and let $\Poly(B,\one)$ be the normalized translate of their Minkowski average constructed in Step~3 of Algorithm~\ref{alg:me}. Proposition~\ref{prop::average} implies that every $\mathbb P_i$ belongs to $\mathfrak P(B,\one)$.

Because $b_i\in\interior{\mathbb M_i}$, each $\mathbb P_i$ is simple and minimally represented by $A$. Hence every ray generated by a row of $A$ is a ray of the normal fan of each $\mathbb P_i$. The normal fan of $\Poly(B,\one)$ is their common refinement and retains these rays. After a row permutation, we may therefore write
\[
B=\begin{pmatrix}D A\\ \widehat B\end{pmatrix},\qquad
D=\operatorname{diag}(\rho_1,\ldots,\rho_n),\quad \rho_i>0.
\]

Let $c\in\mathbb R^n$ satisfy $\Poly(A,c)\neq\varnothing$. Replacing $c$ by its support vector, if necessary, does not change the polytope and ensures $c\in\overline{\mathbb C}(A)$. Since the configuration poset is finite, $\mathbb K(A,c)$ is contained in a maximal element $\mathbb M_k=\mathbb K(A,b_k)$. Thus $c\in\mathbb M_k$, and the normal fan of $\mathbb P_k$ refines that of $\Poly(A,c)$. By the common-refinement property, the normal fan of $\Poly(B,\one)$ also refines that of $\Poly(A,c)$. Define
\[
\widehat c_j\defeq\max_{x\in\Poly(A,c)}B_jx,\qquad j\in\{1,\ldots,m\}.
\]
Since the rows of $B$ include positive multiples of all rows of $A$, the additional inequalities are valid support inequalities and $\Poly(B,\widehat c)=\Poly(A,c)$. The normal-fan characterization gives $\widehat c\in\mathbb K(B,\one)$. Therefore, $B$ is a configuration enclosure of $A$; see~\cite[Thm.~8]{McMullen1973} and~\cite{Houska2025}.

It remains to prove minimality. Let $D\in\mathbb A(r,d)$ be any configuration enclosure of $A$. Then $\mathfrak P(D,\one)$ contains every $\mathbb P_i$. Lemma~\ref{lem::MinkowskiClass}, together with closure under nonnegative scaling, implies that this class contains their Minkowski average. It is also translation invariant because $\mathrm{Im}(D)$ is the lineality space of $\mathbb K(D,\one)$. Hence, it contains the normalized translate $\Poly(B,\one)$ from Step~3. Thus, $D$ provides an $r$-row facet representation of $\Poly(B,\one)$. Since $B$ is a minimal facet representation of this polytope, $m\leq r$. Therefore, no configuration enclosure of $A$ has fewer rows than $B$.
\qed

In particular, the proof shows that the enclosure matrix contains every original normal direction, possibly with a positive rescaling.

\input{figures/fig8}

\begin{example}[A configuration enclosure]\label{ex::cutcube}
The steps of Algorithm~\ref{alg:me} are illustrated by the matrix
\[
A \ = \ \left(
\begin{array}{rrrrrrr}
1 & 0 & 0 & \ -1 &  0 &  0 & 1 \\
0 & \ \phantom{-}1 & 0 &  0 & \ -1 &  0 & 1 \\
0 & 0 & \ \phantom{-}1 &  0 &  0 & \ -1 & \ \phantom{-}1
\end{array}
\right)^\tr.
\]
The first six rows of this matrix are the facet normals of the unit cube in $\mathbb R^3$, while the seventh row vector of $A$ is orthogonal to the convex hull of the first three row vectors of $A$. Consequently, $A$ can be used to represent, for example, the unit cube, and many other geometric objects that can be obtained by intersecting the unit cube with a coordinate-aligned simplex. A more in-depth analysis shows that there are $8$ configuration cones with nonempty interior, corresponding to the $8$ maximal elements of the poset $\mathfrak K(A)$. For example, the vectors
\begin{align*}
b_1 & \ = \ \left( \ 2, \ 2, \ 2, \ 0, \ 0, \ 0, \ 3 \ \right)^\tr \\
b_2 & \ = \ \left( \ 1, \ 3, \ 3, \ 0, \ 0, \ 0, \ 5 \ \right)^\tr, \\
b_3 & \ = \ \left( \ 3, \ 1, \ 3, \ 0, \ 0, \ 0, \ 5 \ \right)^\tr, \\
b_4 & \ = \ \left( \ 3 , \ 3, \ 1, \ 0, \ 0, \ 0, \ 5 \ \right)^\tr, \\
b_5 & \ = \ \left( \ 2 , \ 2, \ 4, \ 0, \ 0, \ 0, \ 5 \ \right)^\tr, \\
b_6 & \ = \ \left( \ 4, \ 2, \ 2, \ 0, \ 0, \ 0, \ 5 \ \right)^\tr, \\
b_7 & \ = \ \left( \ 2, \ 4, \ 2, \ 0, \ 0, \ 0, \ 5 \ \right)^\tr, \\
\text{and} \quad 
b_8 & \ = \ \left( \ 2, \ 2, \ 2, \ 0, \ 0, \ 0, \ 5 \ \right)^\tr
\end{align*}
represent $8$ different polytopes $\Poly(A,b_i)$, $i \in \{ 1, \ldots, 8 \}$. These polytopes are all simple and pairwise
configuration-inequivalent with respect to the fixed row labels of
$A$, although some of them are combinatorially equivalent after
relabeling their facets. Thus, if we initialize Algorithm~\ref{alg:me} with $A$, we can pick these vectors as representatives of the maximal elements of $\mathfrak K(A)$. Next, we work out the Minkowski average in Step~3) of Algorithm~\ref{alg:me}. After scaling the rows of $B$ appropriately, this yields
\[
B \ = \ \left(
\begin{array}{rrrrrrrrrr}
1 & 0 & 0 & \ -1 &  0 &  0 & 1 & 0 & 1 & 1  \\
0 & \ \phantom{-}1 & 0 &  0 & \ -1 &  0 & 1 & 1 & 0 & \ \phantom{-}1 \\
0 & 0 & \ \phantom{-}1 &  0 &  0 & \ -1 & \ \phantom{-}1 & \ \phantom{-}1 & \ \phantom{-}1 & 0
\end{array}
\right)^\tr.
\]
A corresponding (appropriately scaled) representative vector is given by
\[
c^\star \ = \ \left(
\begin{array}{rrrrrrrrrr}
3 & 3 & 3 & 0 & 0 & 0 & 6 & 5 & 5 & 5
\end{array}
\right)^\tr,
\]
where $\mathfrak P(B,c^\star)$ contains the Minkowski average polytope. The translation and positive row scaling used in Step~3 merely normalize an interior representative to $\one$ and do not change these enclosure directions or the represented family. All simple polytopes in the Minkowski class $\mathfrak P(B,c^\star)$ have $10$ facets in total. These facets always consist of $1$ hexagon, $6$ pentagons, and $3$ rectangles; see Figure~\ref{fig:4}.
\end{example}

\begin{remark}
Minimal configuration enclosures need not be monotypic. In Example~\ref{sec::examples1}, the minimal configuration enclosure is monotypic, whereas the minimal configuration enclosure obtained in Example~\ref{ex::cutcube} is not, 
since $\mathfrak{K}(B)$ admits more than one maximal element. In particular, 
the polytope $P(B,\tilde{c})$ with $\tilde{c} = (4,4,4,0,0,0,7,5,5,5)^\intercal$ is simple and minimally represented by all rows of $B$ but 
it is not configuration-equivalent to $P(B,c^\star)$, hence $\sigma(B) > 1$.    
\end{remark}

\subsection{Proof of Theorem~\ref{thm::dense}}
\label{sec::proof1}
Let
\[
\mathfrak C_1^d\defeq\left\{\mathbb X\in\mathfrak C^d\ \middle|\ 0\in\mathbb X,\ \diam(\mathbb X)=1\right\}.
\]
The Blaschke selection theorem implies that $\mathfrak C_1^d$ is compact in the Hausdorff metric; see~\cite[Sec.~1.8]{Schneider2013}. For each $\mathbb X\in\mathfrak C_1^d$, classical polytope-approximation results provide a full-dimensional polytope $\mathbb Q_{\mathbb X}$ satisfying
\[
d_{\mathrm H}(\mathbb X,\mathbb Q_{\mathbb X})<\epsilon/2;
\]
see~\cite{Dudley1974,Bronshtein1975,Schneider1981}. Compactness yields finitely many sets $\mathbb X_1,\ldots,\mathbb X_N\in\mathfrak C_1^d$ whose Hausdorff balls of radius $\epsilon/2$ cover $\mathfrak C_1^d$. Let $A$ collect all outward facet normals of the finitely many polytopes $\mathbb Q_{\mathbb X_1},\ldots,\mathbb Q_{\mathbb X_N}$, after repeated positive directions have been removed. Since each $\mathbb Q_{\mathbb X_k}$ is bounded and full-dimensional, its facet normals positively span $\mathbb R^d$; hence, $A\in\mathbb A(n,d)$ for some $n$. Assigning the additional inequalities their support values gives a representation $\mathbb Q_{\mathbb X_k}=\Poly(A,b_k)$ for every $k$. Hence, for every $\mathbb X\in\mathfrak C_1^d$, some $k$ satisfies
\[
d_{\mathrm H}(\mathbb X,\Poly(A,b_k))
\leq d_{\mathrm H}(\mathbb X,\mathbb X_k)
+d_{\mathrm H}(\mathbb X_k,\mathbb Q_{\mathbb X_k})<\epsilon.
\]

Translation and positive scaling extend this estimate to every $\mathbb Y\in\mathfrak C^d$ with positive diameter: for $y\in\mathbb Y$, apply the normalized result to $(\mathbb Y-y)/\diam(\mathbb Y)$ and use
\[
\alpha\Poly(A,b)+y=\Poly(A,\alpha b+Ay).
\]
Singletons are represented exactly because $\Poly(A,Ay)=\{y\}$. Thus, the family of all polytopes $\Poly(A,\cdot)$ is $\epsilon$-dense in $\mathfrak C^d$. Finally, Theorem~\ref{thm::enclosure} provides a configuration enclosure $B$ of $A$. By definition, $\mathfrak P(B,\one)$ contains every nonempty polytope representable with $A$ and is therefore also $\epsilon$-dense in $\mathfrak C^d$. Renaming $B$ as $A$ proves the theorem.
\qed

\subsection{Computational Scope and Offline Construction}
\label{sec::computational-scope}

Algorithm~\ref{alg:me} has a primarily theoretical role. Together with
Theorem~\ref{thm::enclosure}, it provides a finite construction of a
minimal configuration enclosure for every admissible facet matrix.
As shown in the preceding subsection, this construction also completes
the constructive proof of Theorem~\ref{thm::dense}. Thus,
Algorithm~\ref{alg:me} establishes not only the existence of universal
Minkowski classes, but also that they can, in principle, be generated by
a finite procedure.

To make the enumeration step explicit, quotient the parameter space by the translation lineality $\mathrm{Im}(A)$ and fix one positive scale. The maximal elements of $\mathfrak K(A)$ then appear as the maximal cells of a finite type-cone, or chamber, subdivision. Two maximal elements are adjacent when their normalized cells share a codimension-one face. The resulting adjacency graph is finite and connected: the chamber fan is the normal fan of a secondary polyhedron, and its maximal cones are connected through facet adjacencies~\cite[Cor.~1.42 and Rem.~1.45]{FillastreIzmestiev2017}. A generic parameter in the relative interior of the normalized admissible domain supplies an initial maximal cell. Starting from this cell, one may therefore enumerate all maximal elements by generating the neighboring cells across its irredundant facets and traversing this graph with breadth-first or depth-first search; reverse search provides a memory-efficient alternative~\cite{Avis1996}. This supplies a finite enumeration scheme in principle, although constructing efficient local adjacency oracles and output-sensitive implementations remains beyond the scope of this paper.

This constructive result should not be interpreted as a claim that
combinatorial averaging is the most efficient method for computing
configuration enclosures. The desired object is the enclosure matrix
$B$, rather than the enumeration or the Minkowski average itself.
Consequently, alternative constructions might generate a configuration
enclosure---possibly a minimal one---without explicitly enumerating all
maximal elements of $\mathfrak K(A)$ or forming their complete Minkowski
average.

The relevant computational tasks belong to a broad and active research
area on output-sensitive enumeration methods for combinatorial and
polyhedral problems. Representative developments include pivoting and
reverse-search methods for polyhedra and hyperplane arrangements
\cite{Avis1992,Avis1996}, algorithms and complexity results
for Minkowski sums
\cite{GritzmannSturmfels1993,Fukuda2004}, and region-enumeration
procedures in multiparametric programming
\cite{TondelJohansenBemporad2003}. These methods indicate several
possible directions for specialized implementations of
Algorithm~\ref{alg:me}, but developing and comparing such
implementations is beyond the scope of this article.

The computational burden of Algorithm~\ref{alg:me} has two principal
sources: the enumeration of the $\sigma(A)$ maximal elements of
$\mathfrak K(A)$ and the construction of a minimal facet representation
of their Minkowski average. A coarse bound on the size of the latter
output can nevertheless be obtained. Since all representative
polytopes $\Poly(A,b_i)$ share the same facet matrix $A$, every edge
direction is orthogonal to $d-1$ linearly independent rows of $A$.
Hence, the number of possible unoriented edge directions is at most
\[
N_{\mathrm e}
   \ \leq\
   \binom{n}{d-1}.
\]
Moreover, every edge direction of the Minkowski average is inherited
from an edge direction of one of its summands. Since a facet contains
$d-1$ linearly independent edge directions, its normal line is
determined by at most $d-1$ elements of this set. Therefore, the number
$\overline m$ of candidate facet-normal lines, where opposite normals
are identified, satisfies
\[
\overline m
   \ \leq\
   \binom{N_{\mathrm e}}{d-1}
   \ \leq\
   \binom{\binom{n}{d-1}}{d-1}.
\]
Consequently, without identifying opposite facet normals, the number
$m$ of rows of $B$ admits the conservative bound
\[
m
   \ \leq\
   2\binom{\binom{n}{d-1}}{d-1}.
\]
This is only an output-size estimate. In particular, it does not bound
$\sigma(A)$, the cost of enumerating the maximal configuration cones,
or the possible growth of intermediate representations during the
Minkowski-sum computation. It should therefore not be interpreted as
a complete runtime-complexity bound for Algorithm~\ref{alg:me}.

\input{figures/fig5}

\begin{example}[Dodecahedron]

The conservatism of the preceding estimate is illustrated by
Figure~\ref{fig::5}. The original dodecahedron has dimension $d=3$ and
$n=12$ facets. Thus,
\[
N_{\mathrm e}
   \ \leq\
   \binom{12}{2}
   =66,
\qquad
\binom{N_{\mathrm e}}{2}
   \ \leq\
   \binom{66}{2}
   =2145.
\]
Equivalently, the general oriented-normal estimate gives
\[
m \ \leq\ 2\binom{66}{2}=4290.
\]
In contrast, Algorithm~\ref{alg:me} returns a minimal configuration
enclosure with only
\[
m=62
\]
facets. The enclosure consists of $12$ decagonal, $20$ hexagonal, and
$30$ rectangular facets. This example shows that the combinatorial
estimate may overstate the actual output size by several orders of
magnitude.
\end{example}

Once a suitable enclosure has been computed, the combinatorial
construction need not be repeated in subsequent control applications.
The resulting facet and vertex coefficient matrices can be stored and
reused for different systems of the same dimension, or distributed as
precomputed template libraries. The online polytopic optimization
problems then use only the resulting jointly affine facet--vertex
parameterization and do not invoke Algorithm~\ref{alg:me}.

Randomized exploration, symmetry reduction, incremental Minkowski-sum
updates, and partial averaging may all be useful for reducing the
practical computational effort. However, the exact guarantee of
Theorem~\ref{thm::enclosure} relies on accounting for every maximal
element of $\mathfrak K(A)$; averaging only a sampled subset generally
produces an uncertified candidate rather than a configuration
enclosure. Finding efficient, preferably output-sensitive methods for
computing minimal configuration enclosures therefore remains an open
problem and is beyond the scope of the present manuscript.

\subsection{Open Problem: Adversarial Convex Bodies}
\label{subsec:adversarial-body}

Theorem~\ref{thm::dense} establishes the existence of configuration-constrained
Minkowski classes with arbitrarily small approximation error. For a
prescribed class $\mathfrak{P}(A,b)$, its worst-case approximation
error can be defined as
\[
  \gamma(A,b)
  :=
  \sup_{\substack{\mathbb Y\in\mathfrak C^d\\
                  \diam(\mathbb Y)>0}}
  \inf_{c\in\mathbb K(A,b)}
  \frac{d_{\mathrm H}\bigl(\mathbb Y,\Poly(A,c)\bigr)}
       {\diam(\mathbb Y)}.
\]
Since $\mathfrak{P}(A,b)$ is invariant under translations and
positive scalings, the outer supremum may equivalently be restricted
to sets satisfying
\[
  0\in\mathbb Y
  \qquad\text{and}\qquad
  \diam(\mathbb Y)=1.
\]
This normalized family is compact in the Hausdorff metric, while the
distance to the closed class $\mathfrak{P}(A,b)$ is continuous.
Consequently, the supremum is attained. We call any corresponding
maximizer $\mathbb Y^\star\in\mathfrak C^d$ an adversarial convex body for
the class $\mathfrak{P}(A,b)$.

The geometric properties of such adversarial convex bodies remain
largely open. In particular, it is not clear whether an adversarial
convex body can always be chosen to be polyhedral, whether its
dimension or combinatorial complexity can be bounded in terms of
$A$ and $\mathbb K(A,b)$, or how its exposed directions are related to the
prescribed facet configuration. For a fixed $\mathbb Y$, the inner
minimization can be treated using the joint facet--vertex
parameterization~\eqref{eq::vertices}. The principal
difficulty is the global optimization over $\mathbb Y\in\mathfrak C^d$.
Structural characterizations and certified computational methods for
determining $\gamma(A,b)$ would provide a quantitative basis for
comparing different configuration cones and their configuration
enclosures.

\section{Applications of Universal Minkowski Classes}
\label{sec:applications}

Optimization over compact convex sets arises in shape optimization,
robust design, uncertainty quantification, and learning-based safety;
see, for example, the Galerkin framework in~\cite{Rieger2021} and
set-based approaches to safe learning in
\cite{KollerEtAl2018,LiEtAl2024}. The main focus of this article,
however, is set-theoretic control, where polytopes are routinely used
to represent invariant sets, reachable sets, robust tubes, and
polyhedral control Lyapunov functions.

This section explains how the approximation result of
Theorem~\ref{thm::dense} connects with convex computational methods in
control. The key object is a configuration triple $(A,E,V)$, which
provides a finite-dimensional parameterization in which both facets and
vertices depend affinely on the same decision variables. A broader
account of the corresponding control formulations can be found in
\cite{Houska2025,HouskaCLF2026}.

\subsection{Minkowski-Convex Set Optimization}
\label{sec:applications:convexity}

Consider an optimization problem whose decision variable is a compact
convex set,
\begin{equation}
\begin{aligned}
    \underset{\mathbb Y\in\mathfrak C^d}{\operatorname{minimize}}
    \quad & \mathcal G_0(\mathbb Y)\\
    \operatorname{subject\ to}
    \quad & \mathcal G_j(\mathbb Y)\leq 0,
    \qquad j\in\{1,\ldots,q\}.
\end{aligned}
\label{eq::abstract-set-optimization}
\end{equation}
A functional $\mathcal G:\mathfrak C^d\to\mathbb R$ is called
Minkowski-convex if
\begin{equation}
\mathcal G\bigl(
    \theta\mathbb Y_1+(1-\theta)\mathbb Y_2
\bigr)
\leq
\theta\mathcal G(\mathbb Y_1)
+(1-\theta)\mathcal G(\mathbb Y_2)
\label{eq::Minkowski-convex-functional}
\end{equation}
for all $\mathbb Y_1,\mathbb Y_2\in\mathfrak C^d$ and
$\theta\in[0,1]$. Problem~\eqref{eq::abstract-set-optimization} is
therefore a convex optimization problem in the natural linear structure
induced by Minkowski addition and nonnegative set scaling. Several
relevant examples are collected below.

\begin{example}[Polytopic inclusions]\label{ex::inclusion}
For a fixed direction $p\in\mathbb R^d$, the support function
\[
    h_{\mathbb Y}(p)
    \defeq
    \max_{x\in\mathbb Y}p^\top x
\]
is Minkowski-affine, because 
\[
h_{\theta\mathbb Y_1+(1-\theta)\mathbb Y_2}(p)
=
\theta h_{\mathbb Y_1}(p)
+(1-\theta)h_{\mathbb Y_2}(p).
\]
Consequently, maxima of affine expressions involving support functions
are Minkowski-convex. This observation is particularly useful for
representing set-containment constraints. For example, if
\[
    \mathbb X
    =
    \{x\in\mathbb R^d\mid Hx\leq h\},
\]
then the functional
\[
    \mathcal G_{\mathrm{inc}}(\mathbb Y)
    \defeq
    \max_i
    \left\{
        h_{\mathbb Y}(H_i^\top)-h_i
    \right\}
\]
is Minkowski-convex and satisfies
\[
    \mathcal G_{\mathrm{inc}}(\mathbb Y)\leq0
    \quad\Longleftrightarrow\quad
    \mathbb Y\subseteq\mathbb X.
\]
\end{example}

\begin{example}[Volume functionals]\label{ex::volume} A standard objective to maximize the volume of sets in $\mathbb R^d$ is given by
\[
    \mathcal G_{\mathrm{vol}}(\mathbb Y)
    \defeq
    -\operatorname{vol}(\mathbb Y)^{1/d}.
\]
The Brunn--Minkowski inequality implies~\cite{Gardner2002,Schneider2013} 
\[
\operatorname{vol}
\bigl(
    \theta\mathbb Y_1+(1-\theta)\mathbb Y_2
\bigr)^{1/d}
\geq
\theta\operatorname{vol}(\mathbb Y_1)^{1/d}
+(1-\theta)\operatorname{vol}(\mathbb Y_2)^{1/d}.
\]
Thus, $\mathcal G_{\mathrm{vol}}$ is Minkowski-convex. Notice that the negative volume itself is not, in
general, Minkowski-convex when $d>1$.
\end{example}

\begin{example}[Distance functionals]
For a fixed point $x\in\mathbb R^d$, the set-distance functional
\[
    \mathcal G_x(\mathbb Y)
    \defeq
    d(x,\mathbb Y)
    =
    \min_{y\in\mathbb Y}\|x-y\|
\]
is Minkowski-convex. Indeed, choosing
$y_i\in\mathbb Y_i$ as nearest points to $x$ gives
\[
\begin{aligned}
&d\bigl(
    x,\theta\mathbb Y_1+(1-\theta)\mathbb Y_2
\bigr)
\\
&\qquad\leq
\left\|
    \theta(x-y_1)+(1-\theta)(x-y_2)
\right\|
\\
&\qquad\leq
\theta d(x,\mathbb Y_1)
+(1-\theta)d(x,\mathbb Y_2).
\end{aligned}
\]
More generally, for a fixed compact convex set $\mathbb X$, both the
directed set distance
\[
    \max_{x\in\mathbb X}d(x,\mathbb Y)
\]
and the Hausdorff distance
\[
    d_{\mathrm H}(\mathbb X,\mathbb Y)
\]
are Minkowski-convex as functions of $\mathbb Y$. These functionals
provide natural measures of set-inclusion violation and geometric
approximation error; see also \cite{Rieger2021}.
\end{example}

\begin{example}[Invariance constraints]\label{ex::invariance}
The preceding properties can be combined to formulate robust
control-invariance conditions. Consider
\[
    x^+=Fx+Gu+w,
    \qquad
    u\in\mathbb U,\quad w\in\mathbb W,
\]
where $\mathbb U$ and $\mathbb W$ are compact and $\mathbb U$ is
convex. Define
\begin{equation}
\mathcal G_{\mathrm{rci}}(\mathbb Y)
\defeq
\max_{x\in\mathbb Y}
\min_{u\in\mathbb U}
\max_{w\in\mathbb W}
d(Fx+Gu+w,\mathbb Y).
\label{eq::robust-invariance-functional}
\end{equation}
This functional is Minkowski-convex. To see this, let
\[
    \mathbb Y_\theta
    =
    \theta\mathbb Y_1+(1-\theta)\mathbb Y_2
\]
and write an arbitrary $x_\theta\in\mathbb Y_\theta$ as
\[
    x_\theta=\theta x_1+(1-\theta)x_2,
    \qquad x_i\in\mathbb Y_i.
\]
For controls $u_i\in\mathbb U$ that attain, or arbitrarily closely approximate, the inner minima, define
\[
    u_\theta=\theta u_1+(1-\theta)u_2\in\mathbb U.
\]
For every $w\in\mathbb W$, convexity of the distance under simultaneous
Minkowski combinations yields
\begin{align*}
&d\bigl(
    Fx_\theta+Gu_\theta+w,\mathbb Y_\theta
\bigr)
\\
&\quad\leq
\theta
d\bigl(
    Fx_1+Gu_1+w,\mathbb Y_1
\bigr)
\\
&\qquad+
(1-\theta)
d\bigl(
    Fx_2+Gu_2+w,\mathbb Y_2
\bigr).
\end{align*}
Taking successively the maximum over $w$, the minimum over the control, and the maximum over the state gives
\[
\mathcal G_{\mathrm{rci}}(\mathbb Y_\theta)
\leq
\theta\mathcal G_{\mathrm{rci}}(\mathbb Y_1)
+(1-\theta)\mathcal G_{\mathrm{rci}}(\mathbb Y_2).
\]
Moreover, since the distance is nonnegative, it holds that
\[
    \mathcal G_{\mathrm{rci}}(\mathbb Y)\leq0 
    \ \Longleftrightarrow \
    \left\{ 
        \begin{aligned}
            &\forall x\in \mathbb Y, \, \exists u\in\mathbb U: \ \forall w\in\mathbb W\\
            &Fx+Gu+w \in \mathbb Y. 
        \end{aligned}
    \right.
\]
\end{example}

These examples cover the main ingredients of many set-theoretic control
problems: geometric size objectives, approximation errors, state-set
containment, and robust invariance. Thus, such problems can often be
expressed in the Minkowski-convex form
\eqref{eq::abstract-set-optimization}; see
\cite{Villanueva2024,Houska2025} for further examples.
The remaining difficulty is that optimization over $\mathfrak C^d$ is
infinite-dimensional. The following subsection explains how universal
configuration-constrained Minkowski classes provide finite-dimensional
convex discretizations of these problems.

\input{figures/fig9}

\subsection{Restoring Convexity to Polytopic Parameterizations}
\label{sec:applications:parameterizations}

Problem~\eqref{eq::abstract-set-optimization} is
infinite-dimensional. A natural finite-dimensional approximation is
obtained by fixing a facet matrix $A\in\mathbb R^{n\times d}$ and
parameterizing the decision set as $\mathbb Y=\Poly(A,b)$. This yields
\begin{equation}
\begin{aligned}
    \underset{b\in\mathbb R^n}{\operatorname{minimize}}
    \quad & \mathcal G_0(\Poly(A,b))\\
    \operatorname{subject\ to}
    \quad & \mathcal G_j(\Poly(A,b))\leq 0,
    \qquad j\in\{1,\ldots,q\}.
\end{aligned}
\label{eq::abstract-set-optimization2}
\end{equation}
Although finite-dimensional, this problem is generally non-convex.
Indeed, as $b$ varies, the vertices and face-incidence structure of
$\Poly(A,b)$ may change. Functionals involving both its facet and
vertex representations therefore become piecewise-defined or
disjunctive in the parameter $b$.

This difficulty is removed by restricting the parameters to a
configuration cone. Let
\[
    \mathbb K(A,\one)
    =
    \left\{
        b\in\mathbb R^n
        \,\middle|\,
        Eb\leq0
    \right\}.
\]
Such a finite polyhedral representation follows from the configuration-cone construction; see~\cite{Villanueva2024,Houska2025}.
The resulting discretization is
\begin{equation}
\begin{aligned}
    \underset{b\in\mathbb R^n}{\operatorname{minimize}}
    \quad & \mathcal G_0(\Poly(A,b))\\
    \operatorname{subject\ to}
    \quad & \mathcal G_j(\Poly(A,b))\leq 0,
    \qquad j\in\{1,\ldots,q\},\\
    & Eb\leq0.
\end{aligned}
\label{eq::abstract-set-optimization3}
\end{equation}
On $\mathbb K(A,\one)$, the facets and vertices of $\Poly(A,b)$ admit
a joint affine parameterization. More fundamentally, the set-valued
mapping $b\mapsto\Poly(A,b)$ is affine with respect to Minkowski
combinations:
\begin{equation}
\Poly\bigl(A,\theta b+(1-\theta)c\bigr)
=
\theta\Poly(A,b)+(1-\theta)\Poly(A,c)
\label{eq::Minkowski-affine-template}
\end{equation}
for all $b,c\in\mathbb K(A,\one)$ and $\theta\in[0,1]$. Hence, every
Minkowski-convex functional $\mathcal G_j$ induces a convex function
\[
    b\longmapsto\mathcal G_j(\Poly(A,b))
\]
on the configuration cone. Consequently,
Problem~\eqref{eq::abstract-set-optimization3} is a
finite-dimensional convex optimization problem. For many problems in
set-theoretic control, its objective and constraints additionally admit
explicit linear or conic representations.

Theorem~\ref{thm::dense} shows that imposing a configuration constraint
does not create an intrinsic geometric approximation barrier: suitably
chosen configuration-constrained Minkowski classes can approximate
every compact convex set to arbitrary accuracy. Suppose that the functionals $\mathcal G_j$ are Hausdorff-continuous, the feasible sets lie in a common Hausdorff-compact family, an optimal solution exists, and Problem~\eqref{eq::abstract-set-optimization} admits a strictly feasible solution. Minkowski combinations of an optimal and a strictly feasible solution yield strictly feasible sets with objective values arbitrarily close to the optimum. Density and continuity then imply convergence of the optimal values of increasingly accurate universal-template discretizations. Configuration constraints
thus restore convexity without fundamentally sacrificing geometric
expressiveness.

\begin{example}[CI polytopes: convex formulation]\label{ex::RCI}
\label{sec:applications:numerical}

We illustrate the use of configuration templates by computing control
invariant inner approximations of the maximal control invariant set for
the three-dimensional system
\begin{equation*}
    x^+ = Fx + u  , \quad\text{with} \quad F = \begin{pmatrix}
    0&0&1\\
    \cos{(7)}&-\sin{(7)}&0\\
    \sin{(7)}&\cos{(7)}&0
    \end{pmatrix}^3
\label{eq::numerical-dynamics}
\end{equation*}
where the angle is measured in radians. The state and control constraints are given by
\begin{equation*}
    \mathbb X
    =
    \left\{
        x\in\mathbb R^3
        \, \middle|\,
        x\geq0,\, \one^\intercal x\leq1
    \right\} \ \text{and} \     \mathbb U
    =
    \left[-\frac{1}{10},\frac{1}{10}\right]^3,
\end{equation*}
respectively.

The computation of a maximal control invariant set 
can be formulated as a Minkowski-convex
optimization problem over compact convex subsets of $\mathbb X$. In
this formulation, state-set containment and control invariance are
imposed through the functionals introduced in
Examples~\ref{ex::inclusion} and~\ref{ex::invariance}, respectively,
while the objective is chosen as the volume functional 
\begin{equation*}
\mathcal G_{\rm vol}(\mathbb Y)
= -\operatorname{vol}(\mathbb Y)^{1/3}
\end{equation*}
from Example~\ref{ex::volume}.

As described in the preceding section, a finite-dimensional convex
discretization is obtained by restricting the decision set to a
Minkowski class $\mathfrak{P}(C,c)$. Using the equivalences in
Examples~\ref{ex::inclusion} and~\ref{ex::invariance}, with
$\mathbb W=\{0\}$, the feasible region
\begin{equation*}
\left\{ 
\mathbb{Y} \in \mathfrak{P}(C,c) \,\middle|\, 
\begin{aligned}
&\mathcal{G}_{\mathrm{rci}}(\mathbb Y) \leq 0 \\
&\mathcal{G}_{\mathrm{inc}}(\mathbb Y) \leq 0 
\end{aligned}
\right\}
\end{equation*}
corresponds to the finite-dimensional polyhedron
\begin{equation*}
\mathbb F = \left\{ 
b \in \mathbb{R}^{r_C} \,\middle|\, 
\begin{aligned}
&\exists u_1,\ldots,u_{\nv} \in \mathbb U:\ \forall i\in\{1,\ldots,\nv\}, \\
&C(FV_ib + u_i) \leq b ,\\
&\one^\intercal V_ib \leq 1, \, V_ib \geq 0, \\
&Eb \leq 0
\end{aligned}
\right\}.
\end{equation*}
Here, the vertex matrices $V_1,\ldots,V_{\nv}$ provide the joint
facet--vertex representation of the selected configuration class, which 
is represented by the matrix $E$, compare~\cite[Sect.~3.5]{Villanueva2024}. The first constraint enforces the invariance condition at the vertices, which is equivalent to enforcing
$\mathcal G_{\mathrm{rci}}(\Poly(C,b)) \leq 0$, with $\mathbb W = \{ 0 \}$; see also~\cite{Houska2025}. In this context, the constraint $E b \leq 0$
restricts $b$ to the chosen configuration cone, and the constraints on $u_i$ ensure admissibility of the associated vertex controls. Similarly, $V_i b \geq 0$ and $\one^\intercal V_i b \leq 1$ for all vertices are equivalent to $\Poly(C,b) \subseteq \mathbb X$; that is $\mathcal G_{\mathrm{inc}}(\Poly(C,b)) \leq 0$.

Let $\mathcal T$ be a set of index tuples enumerating the vertices of the 
simplices a of a tetrahedralization of a representative polytope in the
relative interior of the configuration cone.
For each index tuple 
$\tau=(i_0,i_1,i_2,i_3)\in\mathcal T$ of a tetrahedron, order its vertices such that
the corresponding determinant is positive at the representative, and
define
\[
D_{\tau,j}:=V_{i_j}-V_{i_0},
\qquad j\in\{1,2,3\}.
\]
Since the labeled face configuration is constant on the relative
interior, the same vertex-index tetrahedralization applies there and
\[
\operatorname{vol}(\mathbb P(C,b))
=
\frac{1}{6}\sum_{\tau\in\mathcal T}
\det \begin{bmatrix}
D_{\tau,1}b & D_{\tau,2}b & D_{\tau,3}b
\end{bmatrix}.
\]
Thus, $\operatorname{vol}(\mathbb P(C,b))$ is a homogeneous cubic
polynomial in $b$, which extends continuously to the boundary of the
configuration cone. Although this polynomial is not concave in
general, its cube root is concave by the Brunn--Minkowski inequality.

Therefore, for each candidate class, the problem
\begin{equation}\label{eq::numerical-template-LP}
\underset{b\in\mathcal F}{\operatorname{maximize}}
\quad
\operatorname{vol}(\mathbb P(C,b))^{1/3}.
\end{equation}
returns a control invariant set with maximal volume. Here, we solve it for each of the eight maximal configuration cones
\[
\mathbb M_q=\mathbb K(A,b_q),
\qquad q\in\{1,\ldots,8\},
\]
in Example~\ref{ex::cutcube} and its configuration enclosure. Consequently, every feasible solution of
\eqref{eq::numerical-template-LP} defines a certified control invariant
inner approximation of the maximal control invariant set. 
Alternative set-size objectives can also be used to select control invariant
sets while preserving a linear-programming formulation; see, e.g.,
\cite[Sections ~3.1 and 3.2]{Houska2025} for examples applied to reachable and positive invariant polytopes. Such objectives may be useful when a computationally
simpler selection criterion is preferred over the volume-based objective.

Figure~\ref{fig:nested-configurations} shows the resulting control
invariant inner approximations. The lighter enclosing polytope is the
maximal control invariant set, while the darker
polytopes are the inner approximations, of maximal volume, obtained from the eight
maximal configuration cones of $A$. The optimizer associated with
\begin{equation*}
\mathbb M_1=\mathbb K(A,b_1)
\end{equation*}
is the top-left polytope in Figure~\ref{fig:nested-configurations} and has seven facets and ten vertices.

In this example, the optimizer obtained from the
configuration-enclosure class happens to
coincide with the optimizer obtained from the first configuration cone
$\mathbb M_1$. Although a generic polytope in the enclosure class has ten
facets, the optimal parameter lies on the boundary of its
configuration cone. At this boundary point, some facet inequalities
become redundant, and the resulting polytope reduces to the
seven-facet configuration associated with $\mathbb M_1$.
The results in Figure~\ref{fig:nested-configurations} show that for this example optimizing over the configuration enclosure yields a control invariant set whose volume is at least as large as that of the polytopes obtained from the basis configurations.

\end{example}

\section{Conclusion}
This paper has established that prescribing a fixed facet--vertex configuration creates no intrinsic geometric approximation barrier. For every dimension $d\geq1$ and every accuracy $\epsilon>0$, Theorem~\ref{thm::dense} provides a configuration-constrained Minkowski class that is $\epsilon$-dense in the compact convex sets. The combinatorial-averaging construction produces such classes through minimal configuration enclosures while retaining the joint affine facet--vertex parameterization required for convex computation. In set-theoretic control, this resolves a geometric question left open by earlier configuration-constrained tube-MPC and control-Lyapunov-function constructions~\cite{Villanueva2024,HouskaCLF2026}: restricting a parametric polytope to one suitably chosen fixed configuration does not inherently sacrifice universal approximation capability. 

The final section showed how invariant-set and related control problems can
be cast as Minkowski-convex set-optimization problems and discretized through
finite-dimensional convex parameterizations by configuration-constrained
polytopes. More broadly, the theoretical results establish that such
parameterizations can be made arbitrarily expressive and open several
directions for further research, including the efficient construction of
universal Minkowski classes, the structural and computational characterization
of adversarial convex bodies, and the output-sensitive enumeration of maximal
configuration cones.

\section*{Acknowledgement}
The authors would like to thank Colin Jones for inspiring discussions during the first author’s research stay at EPFL in 2023. During these discussions, Colin Jones raised the question  whether configuration-constraints fundamentally obstruct the approximation of convex bodies to arbitrary accuracy. This motivated the central problem addressed in this work. Our results confirm that, provided the configuration templates are chosen appropriately, no such obstruction exists.

\bibliographystyle{plain}
\bibliography{bib/references}

\end{document}